\documentclass[12pt, reqno]{amsproc}
\newtheorem{theorem}{\sc Theorem}[section]
\newtheorem*{theorem*}{\sc Theorem A}
\newtheorem{lemma}[theorem]{\sc Lemma}
\newtheorem{proposition}[theorem]{\sc Proposition}

\newcommand{\al}{\alpha }

\newcommand{\PSL}{\operatorname{PSL}}
\newcommand{\Sz}{\operatorname{Sz}}
\begin{document}
\title[Criteria for solubility and nilpotency]{Criteria for solubility and nilpotency of finite groups with automorphisms}
\author{Cristina Acciarri}

\address{Cristina Acciarri: Dipartimento di Scienze Fisiche, Informatiche e Matematiche, Universit\`a degli Studi di Modena e Reggio Emilia, Via Campi 213/b, I-41125 Modena, Italy}
\email{cristina.acciarri@unimore.it}

\author{Robert M. Guralnick}

\address{Robert M. Guralnick:  Department of Mathematics, University of Southern California, Los Angeles, CA 90089-2532, USA}
\email{guralnic@usc.edu}

\author{Pavel Shumyatsky }
\address{ Pavel Shumyatsky: Department of Mathematics, University of Brasilia,
Brasilia-DF, 70910-900 Brazil}
\email{pavel@unb.br}
\thanks{The authors thank the reviewer for pointing out an error in an earlier version of the paper. The second author was  partially supported by the NSF grant DMS-1901595 and a Simons Foundation Fellowship 609771. The third author was supported by CNPq and FAPDF, Brazil.}
\keywords{Finite groups, automorphisms}
\subjclass[2010]{20D45}

\begin{abstract} 
Let $G$ be a finite group admitting a coprime automorphism $\al$. Let $J_G(\al)$ denote the set of all commutators $[x,\al]$, where $x$ belongs to an $\al$-invariant Sylow subgroup of $G$. We show that $[G,\al]$ is soluble or nilpotent if and only if any subgroup generated by a pair of elements of coprime orders from the set $J_G(\al)$ is soluble or nilpotent, respectively.
\end{abstract}

\maketitle

\section{Introduction}
In \cite{BW} Baumslag and Wiegold established  the following sufficient condition for the nilpotency of a finite group $G$.
\begin{theorem}
Let $G$ be a finite group in which $|ab|=|a||b|$, whenever the elements $a,b$ have coprime orders. Then $G$ is nilpotent.
\end{theorem}
Here the symbol $|x|$ stands for the order of an element $x$ in a group $G$. Obviously the condition above is also necessary for the nilpotency of $G$.  We mention that there are several  recent results related to the theorem of Baumslag and Wiegold (see for example \cite{BS,BMS,GM,MS,daSilvaS}).

An automorphism $\alpha$ of a finite group $G$ is said to be coprime if $(|G|,|\alpha|)=1$. Following \cite{AGS} denote by  $I_G(\alpha)$ the set of commutators $g^{-1}g^{\alpha}$, where $g\in G$. Let $[G,\alpha]$ be the subgroup generated by $I_G(\alpha)$. In \cite{AGS} the authors studied the impact of $I_G(\alpha)$ on the structure of $[G,\alpha]$.  

Here we establish the following variation of the result of Baumslag and Wiegold.

\begin{theorem} \label{nilpotency2} Let $G$ be a finite group  admitting a  coprime automorphism $\al$. Then $[G,\al]$ is nilpotent if, and only if, $|xy|=|x||y|$ whenever $x$ and $y$ are elements of coprime prime power orders from $ I_G(\al)$.
\end{theorem}

At the start of this project, we did not know whether the hypothesis on the orders of elements in Theorem \ref{nilpotency2} is inherited  by quotient groups. In order to overcome this issue we work instead with a somewhat different condition that behaves well with respect to forming quotients. Let $J_G(\al)$ denote the set of all commutators $[x,\al]$, where  $x$ belongs to an $\al$-invariant Sylow subgroup of $G$. Observe that $J_G(\al)$ is a subset of $I_G(\al)$ and the elements of $J_G(\al)$ have prime power order. Moreover  note that $J_G(\al)$ is a generating set for $[G,\al]$. Indeed this easily follows from \cite[Lemma 2.4]{Shumy}. It turns out that properties of $G$ are pretty much determined by those of subgroups generated by elements of coprime orders from $J_G(\al)$.

It is well known that if any pair of elements of a finite group generates a soluble (respectively nilpotent) subgroup, then the whole group is soluble (respectively nilpotent).  One of the theorems established in \cite{AGS} provides a variation of this for groups with automorphisms.
\bigskip

{\it Let $G$ be a finite group admitting a coprime automorphism $\al$. If any pair of elements from $I_G(\al)$ generates a soluble subgroup, then $[G,\al]$ is soluble. If any pair of elements from $I_G(\al)$ generates a nilpotent subgroup, then $[G,\al]$ is nilpotent.
}
\bigskip

Here this will be extended as follows.
\bigskip

{\it
 Let $G$  be a finite group admitting  a coprime automorphism $\al$.  Then $[G,\al]$ is soluble if and only if any subgroup generated by a pair of  elements of coprime orders from $J_G(\al)$ is soluble.
 }
\bigskip

In fact, we will establish a stronger result. Assume that a finite group $G$ admits an automorphism group $A$. If $\al \in A$, let $J_{G,A}(\al)$ be the set of
commutators $[x,\al]$ for $x$ in an $A$-invariant Sylow subgroup of $G$.
 
\begin{theorem}\label{soluble}
Let $G$  be a finite group admitting a coprime group of automorphisms $A$.  If $\al \in A$, then $[G,\al]$ is soluble if and only if any subgroup generated by a pair of  elements of coprime orders from $J_{G,A}(\al)$ is soluble.
\end{theorem}

Note that this fails if the coprimeness assumption is omitted. For example if $\al$ is a transposition in the symmetric group $G=S_n$, any pair of elements from $I_G(\al)$ generates a soluble subgroup while $[G,\al]$ is insoluble for $n\geq5$. 

Theorem \ref{soluble} is used to establish the following related necessary and sufficient conditions for the nilpotency of $[G,\al]$. 

\begin{theorem}\label{nilpotency-2gen}
Let $G$  be a finite group admitting a coprime group of automorphisms $A$, and let $\al\in A$. Then the following statements are equivalent.
\begin{itemize}
\item[(i)] The subgroup $[G,\al]$ is nilpotent;
\item[(ii)] Any subgroup generated by a pair of  elements of coprime orders from $J_{G,A}(\al)$ is nilpotent;
\item[(iii)]  Any subgroup generated by a pair of elements of coprime orders from $J_{G,A}(\al)$ is abelian;
\item[(iv)]  If $x$ and $y$ are elements of coprime orders from $J_{G,A}(\al)$, then $|xy|=|x||y|$;
\item[(v)] If $x$ and $y$ are elements of coprime orders from $J_{G,A}(\al)$, then $\pi(xy)=\pi(x)\cup\pi(y)$.
\end{itemize}
\end{theorem}
Here, $\pi(g)$ denotes the set of prime divisors of the order of $g\in G$. 

Note that Theorem \ref{nilpotency2} easily follows from Theorem \ref{nilpotency-2gen}. 

\section{A condition for solubility of $[G,\al]$}
All groups considered in this paper are finite. We start with a collection of well-known facts about coprime automorphisms of finite groups (see for example \cite{gore}).

\begin{lemma}\label{20} Let a group $G$ admit a coprime group of automorphisms $A$. The following conditions hold:
\begin{itemize} 
\item[(i)] $G=[G,A]C_G(A)$;
\item[(ii)] If $N$ is any $A$-invariant normal subgroup of $G$, we have $C_{G/N}(A)=C_G(A)N/N$; 
\item[(iii)] If $\al\in A$ and $N$ is an $A$-invariant normal subgroup of $G$, we have $J_{G/N,A}(\al)=\{gN \mid g\in J_{G,A}(\al)\}$;
\item[(iv)] If N is any $A$-invariant normal subgroup of $G$ such that $N=C_N(A)$, then $[G,A]$ centralizes $N$;
\item[(v)] Any $A$-invariant $p$-subgroup of $G$ is contained in an $A$-invariant Sylow $p$-subgroup.
\end{itemize}
\end{lemma}

The purpose in this section is to establish conditions that lead to the solubility of $[G,\al]$. 
Throughout, by a simple group we mean a nonabelian simple group. If a simple group $G$ admits a coprime automorphism $\alpha$ of order $e$, then $G=L(q)$ is a group of Lie type and $\al$ is a field automorphism. Furthermore, $C_G(\al)=L(q_0)$ is a group of the same Lie type (and rank) defined over the subfield such that $q=q_0^e$ (see \cite{GLS3}). 

In what follows, $q=p^s$ is a power of a prime $p$. For any positive integer $n$, we say that a prime $r$ is a {\em primitive prime divisor} of $q^n -1$ if $r$ divides $q^n -1$ and $r$ does not divide $q^k -1$ for any positive integer $k< n$. Primitive prime divisors of $q^n +1$ are defined in a similar way.  The following result of Zsigmondy (\cite{Z}) on the existence of primitive prime divisors will be used in the sequel. 

\begin{theorem}[\cite{Z}] \label{Zsigmondy} Let $a>b>0$, $\mathrm{gcd}(a,b)=1$ and $n>1$ be positive integers. Then

\begin{itemize}
\item[(i)] $a^n-b^n$ has a prime divisor that does not divide $a^k-b^k$ for all positive integers $k<n$, unless $a=2,b=1$ and $n=6$; or $a+b$ is a power of $2$ and $n=2$.
\item[(ii)] $a^n+b^n$ has a prime divisor that does not divide $a^k+b^k$ for all positive integers $k<n$, with exception $2^3+1^3$.
\end{itemize}
\end{theorem}

Throughout, the term “semisimple group” means direct product of simple groups.
We are now ready to deal with  the proof of Theorem \ref{soluble}.  

\begin{proof}[Proof of Theorem \ref{soluble}]
If $[G, \al]$ is soluble, the result is clear.  Let us prove  the converse. We wish to show that there are $A$-invariant subgroups $P$ and $Q$ of coprime prime power orders such that $[x,\al]$ and $[y,\al]$ generate an insoluble subgroup for some $x \in P$ and $y \in Q$. Note that any $A$-invariant $p$-subgroup $P$ of $G$ is contained
in an $A$-invariant Sylow $p$-subgroup. 

Suppose that this is false and let $G=[G,A]$ be a counterexample of minimal order.  So any pair of elements of coprime order in $J_{G,A}(\al)$ generate a soluble group
but $[G,\al]$ is insoluble. Suppose that $G$ contains a nontrivial characteristic soluble subgroup $M$. Then, by minimality, the result holds for $A$ acting on $G/M$ and so $[G,\al]M/M$ and $[G,\al]$ are both soluble, a contradiction.

So we may assume that $G$ contains no soluble nontrivial normal subgroups.    Then $F^*(G)$ is a direct product of nonabelian simple groups
and $A$ acts faithfully on $F^*(G)$  (since $F^*(GA)=F^*(G)$).    Choose a component $S$ of $F^*(G)$ so that $\al$ does not centralize $S$.
Let $T=S_1 \times \ldots \times S_t$ be the $A$-orbit of $S=S_1$.  Then $[T,\al]$ is insoluble and so it suffices to assume that $G=T$.

Let $A_1=N_A(S)$.   Then $A_1/C_A(S)$ is a coprime
group of automorphisms of the simple group $S$.  If $A_1$ is trivial, then $|A|=t$ and $A$ acts on $G$ by permuting the coordinates
of $S_1 \times \ldots \times  S_t$.   By \cite[Theorem 1.2]{GT} (or \cite{DGHP}), there exist distinct primes $p$ and $q$, a Sylow $p$-subgroup $P_1$
of $S=S_1$ and a Sylow $q$-subgroup $Q_1$ of $S$ so that $\langle x_1, y_1 \rangle$ is insoluble, with $x_1 \in P_1$ and $y_1 \in Q_1$. 
Then the (direct) product $P$ of the distinct $A$-conjugates of $P_1$ is  an $A$-invariant Sylow $p$-subgroup of $G$.  Similarly the direct product $Q$ of the $A$-conjugates of $Q_1$ is 
an $A$-invariant Sylow $q$-subgroup of $G$.    Note that $[x_1, \al] \in [P,\al]$ and $[y_1,\al] \in [Q,\al] $ generate an insoluble group.

So we may assume that $A_1 \ne 1$.  It follows that $S=L(q)$ is a group of Lie type over the field of $q=p^s$ and $A_1$ induces a cyclic group of field automorphisms (see \cite{GLS3}).

First assume that $S$ has rank at least $2$.  Note that $A_1$ normalizes a Borel subgroup $B$ and, by the structure of field automorphisms, $A_1$ normalizes each parabolic subgroup containing $B$ and indeed $A_1$ acts faithfully on a Levi subgroup $L_1$ of $S$.   Since $q \ge 8$, $L_1$ is insoluble.
It follows that $A$ acts faithfully on the (direct) product $H$ of the $A$-conjugates of $L_1$ and $[H,\al]$ is a nontrivial normal subgroup of $H$. Thus $[H,\al]$ is insoluble and the result follows by induction.  

So assume that $S$ has (twisted) Lie rank $1$. If $S=\mathrm{PSU}_3(q)$ or $^{2}\mathrm{G}_2(q)$, we observe that there is an $A_1$-invariant subgroup isomorphic to 
$\PSL_2(q)$ which is not centralized by $A_1$ and so we can reduce to the case $S=\PSL_2(q)$. Thus, we only need to consider $S=\PSL_2(q)$ with $q=p^s$ for $s$ odd and $s\ge 5$ or $S = \Sz(q)$ with $q = 2^s$ for odd $s >1$.

If $S=\PSL_2(q)$,   take $U$ to be an $A_1$-invariant  Sylow $p$-subgroup  of $S$.  Let now $r$ be  a primitive  prime divisor of $q+ 1$, i.e.  $r$ does not divide $p^i + 1$ for $i < s$  (that always exists by Theorem \ref{Zsigmondy}),  and let  $R$ be an $A_1$-invariant Sylow $r$-subgroup of $S$.   Then $U$ and $R$ are contained in $A$-invariant Sylow subgroups (by taking the product of the distinct $t$ conjugates under $A$). 

Note that any nontrivial element $x$ of $R$ and any nontrivial element $y$ of $U$ generate $S$  by \cite[Theorem 6.25 in Chap.\ 3]{Sz2} (and in particular generate an insoluble subgroup). Moreover $[U,\al]$ and $[R,\al]$ are both contained in $J_{G,A}(\al)$. If $\al$ is contained in $A_1$, then it induces a nontrivial automorphism on $S$ and $[U,\al]\neq 1$.  Further, since $r$ does not divide the order of $C_S(\al)$ we have $[R, \al]=R$. If $x\in [U,\al]$ and $y\in [R,\al]$, then the pair $x,y$ is as required.
If $\al \notin A_1$, then the pair $[x,\al]$ and $[y,\al]$ is as required, since $\al$ conjugates $S$ to some other component $S_i$.

It remains to consider the case $S=\Sz(q)$, where $q=2^s$ for odd $s>1$.   See \cite{Sz1} for properties of Suzuki groups.   Essentially the same argument applies. 
The order of $S$ is $q^2(q-1)(q^2+1)$.  Observe that the maximal subgroups  of $S$ are (up to conjugacy) a Borel subgroup  of order $q^2(q-1)$, a dihedral subgroup of order $2(q-1)$, subfield subgroups, and two subgroups of the form $T.4$, where $T$ is cyclic of order $q  \pm  l + 1$  with $l^2 = 2q $, i.e. of order  $2^s \pm 2^{(s+1)/2} + 1$. Note 
that $(q+l+1)(q-l+1)=q^2+1$.

Let $r$ be a primitive prime divisor of $q^2+1=2^{2s} + 1$. Let $R$ be an $A_1$-invariant Sylow $r$-subgroup. Let $u$ be a primitive prime divisor of $q-1=2^s-1$. Then $u$ does not divide the order of any subfield subgroup and also $u$ does not divide $q^2+1$. By \cite[Theorem 9]{Sz1}, there is no proper subgroup of $S$ whose order is divisible by $ru$. Let $U$  be an $A_1$-invariant Sylow $u$-subgroup of $S$.  Neither of $R$ and $U$ intersects $C_S(\al)$, whence $[R,\al]=R$ and $[U,\al]=U$. Moreover  any element in $R$ or $U$ is a commutator with $\al$.   It follows that $S$ is generated by nontrivial $[x,\al]$ and $[y,\al]$ with $x \in R$ and $y \in U$.   This gives the result if $\al \in A_1$. If $\al$ does not normalize $S$, then, as in the previous case, take  elements $x\in R$ and $y\in U$. The commutators $[x,\al]$ and $[y,\al]$ generate an insoluble group. 
This completes the proof. 
\end{proof}

As a by-product of the proof of Theorem \ref{soluble} we obtain the following proposition.

\begin{proposition}\label{pi with solubility}
Let $G$  be a finite group admitting a coprime group of automorphisms $A$.  Let $\al \in A$, and assume that $\pi(xy)=\pi(x)\cup\pi(y)$ whenever $x$ and $y$ are elements of coprime orders from $J_{G,A}(\al)$.  Then $[G,\al]$ is soluble.
\end{proposition}

The proof will be omitted since it is just an obvious modification of the argument used in Theorem \ref{soluble}. The proposition will be used in the proof of Theorem \ref{nilpotency-2gen}, which in particular shows that under the hypotheses of the proposition the subgroup $[G,\al]$ is nilpotent.

\section{Criteria for nilpotency of $[G,\al]$}

In this section we aim to establish Theorem \ref{nilpotency-2gen} that gives necessary and sufficient conditions for the nilpotency of $[G,\al]$. For the reader’s convenience we restate it here.
\bigskip

{\it Let $G$  be a finite group admitting a coprime group of automorphisms $A$, and let $\al\in A$. Then the following statements are equivalent.
\begin{itemize}
\item[(i)] The subgroup $[G,\al]$ is nilpotent;
\item[(ii)] Any subgroup generated by a pair of  elements of coprime orders from $J_{G,A}(\al)$ is nilpotent;
\item[(iii)]  Any subgroup generated by a pair of elements of coprime orders from $J_{G,A}(\al)$ is abelian;
\item[(iv)]  If $x$ and $y$ are elements of coprime orders from $J_{G,A}(\al)$, then $|xy|=|x||y|$;
\item[(v)] If $x$ and $y$ are elements of coprime orders from $J_{G,A}(\al)$, then $\pi(xy)=\pi(x)\cup\pi(y)$.
\end{itemize}
}

\begin{proof}[Proof of Theorem \ref{nilpotency-2gen}] Note that obviously  (i) implies (ii). Now if $x$ and $y$  are elements of coprime prime power  orders  from $J_{G,A}(\al)$ generating a nilpotent subgroup, then $x$ and $y$ commute with each other and so (ii) implies (iii). If $x$ and $y$  are elements of coprime prime power orders  from $J_{G,A}(\al)$ generating an abelian subgroup, then $|xy|=|x||y|$, and so (iii) implies (iv). Moreover it is immediate to see that (iv) implies (v). 

We now prove the implication (v) $\Rightarrow$  (i). Assume (v). In view of Proposition \ref{pi with solubility}, the subgroup $[G,\al]$ is soluble. Suppose that the result is false and let $G$ be a counterexample of minimal order. Note that $G=[G,\al]^A$, that is, there are no proper $A$-invariant subgroups containing $[G,\al]$. Let $M$ be a minimal $A$-invariant normal subgroup of $G$. Hence, $M$ is an elementary abelian $p$-group,  for some prime $p$, because $G$ is soluble. Since $G/M$ satisfies the hypothesis, by  induction $[G,\al]M/M$ is nilpotent. Taking into account that $G=[G,\al]^A$ we deduce that $G/M$ is nilpotent.

It follows that $G=MQ$, where $Q$ is an $A$-invariant Sylow $q$-subgroup of $G$ satisfying $Q=[Q,\al]^A$. Remark that by definitions of $J_{G,A}(\al)$ and $J_{G}(\al)$ we have $J_{M,A}(\al)=J_{M}(\al)$ and $J_{Q,A}(\al)=J_Q(\al)$.

By hypothesis, for any $x\in J_Q(\al)$ and $y\in J_M(\al)$, the order  of $xy$ is divisible by both $p$ and $q$. Note that the element $y$ is of the form $[m,\al]$, for suitable element $m\in M$  and so $y\in [M,\al]$. If $y\in [M,x]$, then $xy$ is a $q$-element, contradicting the hypothesis. Hence we have $[M,x]\cap[M,\al]=1$, whenever  $x\in J_Q(\al)$.  

Suppose that for a nontrivial element $u$ there is $x\in J_Q(\al)$ such that $u\in C_{M}(x\al^{-1})$ while $u\not\in C_{M}(\al)$. Then the element $[u,\al]$ is nontrivial and belongs to the intersection $[M,\al]\cap[M,x]$. This leads to a contradiction with the above assumption that $[M,x]\cap[M,\al]=1$, whenever $x\in J_Q(\al)$. 

Therefore 
\begin{equation}\label{eq1}C_M(x\al^{-1})\leq C_M(\al)\,\, \text{whenever}\,\, x\in J_Q(\al). \end{equation} Now, let $g$ be an arbitrary element of $Q$ and apply (\ref{eq1})  with $x=[g,\al]$. Then $x\al^{-1}=\al^{-g}$ and we obtain that $C_M(\al)^g=C_M(\al)$. Since this holds for any $g\in Q$, we conclude that $C_M(\al)$ is normal in $G$. Lemma \ref{20} (iv) now shows that $C_M(\al)$ commutes with $[Q,\al]$. Since $G$ is a counterexample of minimal order, we deduce that $C_M(\al)=1$. Thus, we have  $M=[M,\al]$. Recall that $[M,x]\cap[M,\al]=1$. It follows that $[M,x]=1$ for any $x\in J_Q(\al)$ and, taking into account that $Q=[Q,\al]^A$, we obtain that $M\leq Z(G)$, a contradiction.  This proves that (v) implies (i) and thus the theorem is established.
\end{proof}


\end{document}